\documentclass{amsart}

\usepackage{amssymb}
\usepackage{amsmath}
\usepackage{amsfonts}
\usepackage{geometry}
\usepackage{bbm}
\usepackage{hyperref}
\usepackage{tikz}
\usepackage{mathrsfs}
\usepackage{multirow}
\usetikzlibrary{matrix,arrows,decorations.pathmorphing}
\usepackage{verbatim }
\usepackage{mathtools}
\usepackage[author={Sebastian}]{pdfcomment}
\usepackage{todonotes}

\newcounter{cprop}[section]

\newtheorem{theorem}[cprop]{Theorem}

\theoremstyle{plain}

\newtheorem{corollary}[cprop]{Corollary}

\newtheorem{lemma}[cprop]{Lemma}
\newtheorem{proposition}[cprop]{Proposition}

\newtheorem{assumption}[cprop]{Assumption}

\numberwithin{equation}{section}

\theoremstyle{definition}
\newtheorem{definition}[cprop]{Definition}

\theoremstyle{remark}
\newtheorem{remark}[cprop]{Remark}

\renewcommand{\P}{\mathbb{P}}

\newcommand{\R}{\mathbb{R}}
\newcommand{\N}{\mathbb{N}}

\newcommand{\vertiii}[1]{{\left\vert\kern-0.25ex\left\vert\kern-0.25ex\left\vert #1 
		\right\vert\kern-0.25ex\right\vert\kern-0.25ex\right\vert}}

\begin{document}
	\title[Center manifolds]{A general center manifold theorem on fields of Banach spaces}
	
	\author{M. Ghani Varzaneh}
	\address{Mazyar Ghani Varzaneh\\
		Fakult\"at f\"ur  Mathematik und Informatik, FernUniversit\"at in Hagen, Hagen, Germany}
	\email{mazyar.ghanivarzaneh@fernuni-hagen.de}

	\author{S. Riedel}
	\address{Sebastian Riedel \\
		Fakult\"at f\"ur  Mathematik und Informatik, FernUniversit\"at in Hagen, Hagen, Germany}
	\email{sebastian.riedel@fernuni-hagen.de }

	
	\begin{abstract}
		A general local center manifold theorem around stationary trajectories is proved for nonlinear cocycles acting on measurable fields of Banach spaces. Applications to several stochastic (partial) differential equations are given. 
	\end{abstract}
	
	\maketitle
	
	\section*{Introduction}
	
	The center manifold is an object that helps to describe the evolution of a dynamical system. A center manifold is invariant under the flow, and it does not either grow or decay exponentially around a stationary point. The theory of center manifolds is an essential tool for analyzing the stability of systems and, more importantly, in bifurcation theory. Indeed, center manifolds and the idea of normal forms are two standard approaches for simplifying dynamical systems that can be used to reduce the dimension and to eliminate nonlinearities of the system \cite{Car81}. \smallskip
	
	In bifurcation theory, the main purpose is to study the changes in the qualitative or topological structure of the family of flows. The systematic procedure is to compute or approximate the center manifold and then to make a reduction into the center manifold by applying some near identity transformation to simplify the system \cite{Wig03, GH83}. Due to the importance of this theory in applications, there are a lot of results devoted to the existence of center manifolds for a variety of equations and dynamics in both finite and infinite dimensions \cite{VF13, MR18,HV07,GW13}. \smallskip
	
	Center manifold theory can also be expressed for random dynamical systems (RDS), i.e. systems that are also affected by a random signal. Typical examples that induce random dynamical systems are stochastic ordinary differential equations (SDEs) or stochastic partial differential equations (SPDEs) \cite{Arn98}. In the case of SDEs, center manifolds were studied e.g. in \cite{AB89,Box89,Box91,Rob08,NK21}. For SPDEs, the flow takes values in an infinite dimensional space. In this case, random center manifolds are considered e.g. in \cite{DW14, CRD15,CRD19,KN23}. \smallskip
	
	The main goal of this work is to provide general conditions under which a given random dynamical system admits a local center manifold around a stationary point. Our main result, cf. Theorem \ref{LCMT}, is formulated in a generality that allows to apply it to SDEs, SPDEs, and even more general equations provided that they generate a random dynamical system. In fact, in the literature, most of the works that deal with the existence of random invariant manifolds for S(P)DEs face two major challenges:
	\begin{enumerate}
		\item Prove that a given equation generates an RDS.
		\item Prove that this particular RDS admits a center manifold.
	\end{enumerate}
	While the methods to prove point 1 often differ significantly from each other (for example, one can use perfection theorems for crude cocycles \cite{Arn98}, find random mappings that transform the equation to a pathwise solvable random ODE / PDE \cite{CRD15} or apply a pathwise stochastic calculus like rough paths theory \cite{NK21,KN23}), the idea to prove point 2 are often similar. Nonetheless, from a technical point of view, dealing with random dynamical systems in Banach spaces is often very challenging. In fact, the main difficulty arises from the lack of an inner product, which makes it impractical to use the exterior algebra. This, in turn, complicates the interpretation of the volume growth of the Lyapunov exponents for the cocycle defined in a general Banach space. Therefore, providing an abstract theorem that can be applied in great generality will be very useful.\smallskip
	
	Let us discuss a few features of our main abstract result which is formulated in Theorem \ref{LCMT} and how it is related to other results in the literature.
	%abstract center manifold theorem (cf. Theorem \ref{LCMT})
	\begin{itemize}
		\item Since our goal is to apply our main result to SPDEs, it will be necessary that Theorem \ref{LCMT} is formulated for RDS defined on infinite dimensional spaces. There are some works that prove center manifold theorems on RDS that are defined on Hilbert spaces, cf. \cite{CRD15, CRD19}. However, more modern pathwise solution theories for SPDEs (like rough paths theory \cite{GH19,GHT21}) define solutions that take values in a Banach, not only in a Hilbert space. In fact, Theorem \ref{LCMT} can be applied to cocycles defined on separable Banach spaces. With this generality, we can cover, for instance, the results obtained in \cite{KN23}. We provide two general examples of this type of equations in Section \ref{A_3} and Section \ref{A_4}. Furthermore, we discuss an example of a rough different equation in Section \ref{A_2} to which our finding can be applied. In particular, we obtain the results of \cite{NK21} in a more general setting.
		
		\item Some stochastic differential equations induce RDS that are defined in random spaces. This can be the case, for instance, when considering the linearization of an SDE evolving on a manifold \cite{Arn98}. Other examples are given by singular stochastic delay differential equations \cite{GVRS22, GVR21}. In this case, the random spaces have a fiber-type structure. In the infinite dimensional case, one cannot expect that the random spaces are isomorphic, and the correct structures to consider here are \emph{measurable fields of Banach spaces}. Therefore, we formulated Theorem \ref{LCMT} for RDS acting on a measurable field. To our knowledge, this is the first time that a center manifold theorem was formulated for an RDS acting on a field of Banach spaces. However, we want to point out that every fixed Banach space forms a (trivial) field of Banach spaces, thus our theorem covers RDS acting on separable Banach spaces as a special case. A concrete example of a stochastic functional differential equation is discussed in Section \eqref{A_1}. Our results for these equations can even be applied in a more general setting, i.e., when the equation is perturbed by a fractional Brownian motion. However, we focused on the Brownian case to keep our presentation simple.
		
		\item Many works dealing with center manifolds concentrate on the case where the equation has a deterministic fixed point (see e.g \cite{CRD15, NK21, KN23}). However, there are many equations that admit \emph{random} fixed points (or stationary points) only. This holds, for instance, in case of additive noise (an easy example would be the Langevin equation). Therefore, Theorem \ref{LCMT} is formulated for random fixed points. One cornerstone of our result is the Multiplicative Ergodic Theorem (MET) that ensures the existence of Lyapunov exponents for a given compact linear cocycle acting on a field of Banach spaces. This theorem was proved in \cite{GVRS22, GVR23}.
	\end{itemize}
	The article is structured as follows. In Section \ref{sec_inf_mfd}, we introduce our setting and give some background. In Section \ref{main}, after proving some auxiliary Lemmas, we prove our main result, cf. Theorem \ref{LCMT}. Finally, in Section \ref{app}, we give several examples to illustrate how our result can be applied to a large family of random(stochastic) equations.
	\section{Notations and background}\label{sec_inf_mfd}
	In this section, we collect some notations and present background about the \emph{Multiplicative Ergodic Theorem} (MET). Let us first fix the notations.
	
	\begin{itemize}
		\item For a Banach space $(X, \|\cdot\|_X)$, we will usually drop the subindex $\|\cdot\|_X$ and use the symbol $\| \cdot \|$ rather. If  $\psi \colon X \to Y$ is a linear map between Banach spaces, $\| \psi \|$ denotes the  operator norm, i.e.
		\begin{align*}
			\| \psi \| = \sup_{x \in X \setminus \{0\}} \frac{\| \psi(x) \|}{\|x\|}.
		\end{align*}
		\item Let $(X, \|\cdot\|)$ be a Banach space and $Y\subset X$ a closed linear subspace of $X$. Assume $x\in X$. Then by $d(x,Y)$, we denote the usual distance function, i.e.
		\begin{align*}
			d(x,Y):=\inf_{y\in Y}\Vert x-y\Vert.
		\end{align*}
		% \item  Let $ (X, \|\cdot\|)$ be a Banach spaces. For $ x_{1},...,x_{k}\in X $, define % the  of $ (x_{1},x_{2},...,x_{k}) $ by
		% \begin{align}\label{VOL}
			% 	\operatorname{Vol}(x_{1},x_{2},...,x_{k}) := \Vert x_{1}\Vert\prod_{i=2}^{k}d(x_{i},\langle x_{j} \rangle_{1\leq j<i}),
			% \end{align}
		% where $\langle x_{j} \rangle_{1\leq j<i}$ is the linear subspace, generated by $ x_{1},...,x_{i-1}$.
		\item  Let $X$ be a Banach space and $E,F$ two closed subspaces of $X$ such that $E\cap F = \lbrace 0\rbrace$. The linear map $\Pi_{E||F}:=E\oplus F\rightarrow E $ defined by  $(e,f)\rightarrow e$ is called \emph{projection}. In this case, the operator norm is given by
		\begin{align*}
			\| \Pi_{E||F} \| = \sup_{e \in E \setminus\{0\}, f \in F \setminus \{0\}} \frac{\|e\|}{\|e + f\|}.
		\end{align*}
		\item Let \((\Omega, \mathcal{F})\) be a measurable space and \(\{E_\omega\}_{\omega \in \Omega}\) be a family of Banach spaces. Recall that \(\prod_{\omega \in \Omega} E_\omega\) denotes the set of functions \(f: \Omega \to \bigcup_{\omega \in \Omega} E_\omega\) such that \(f(\omega) \in E_\omega\) for every \(\omega \in \Omega\). These functions are also called \emph{sections}. We call  \(\{E_\omega\}_{\omega \in \Omega}\) a \emph{measurable field of Banach spaces} if there exists a set of sections
			\[
			\Delta \subset \prod_{\omega \in \Omega} E_{\omega}
			\]
		with the following properties:
		\begin{itemize}
			\item[(i)] $\Delta$ is a linear subspace of $\prod_{\omega \in \Omega} E_{\omega}$.
			\item[(ii)] There is a countable subset $\Delta_0 \subset \Delta$ such that for every $\omega \in \Omega$, the set $\{g(\omega)\, :\, g \in \Delta_0\}$ is dense in $E_{\omega}$.
			\item[(iii)] For every $g \in \Delta$, the map $\omega \mapsto \| g(\omega) \|_{E_{\omega}}$ is measurable.
		\end{itemize}
		We will most of the time omit the subindex $E_{\omega}$ and just write $\| \cdot \|$ instead of $\| \cdot \|_{E_{\omega}}$ when it is clear from the context which norm is meant.
		\item Let $(\Omega,\mathcal{F})$ be a measurable space. Assume  $\theta \colon \Omega \to \Omega$, $\omega \mapsto \theta \omega$ is a measurable map with a measurable inverse $\theta^{-1}$. In this case, we call $(\Omega,\mathcal{F}, \theta)$ a \emph{measurable dynamical system}. We will use $\theta^n \omega$ for $n$-times applying $\theta$ to an element $\omega \in \Omega$. We also set $\theta^0 := \operatorname{Id}_{\Omega}$ and $\theta^{-n} := (\theta^n)^{-1}$. If $\P$ is a probability measure on $(\Omega,\mathcal{F})$ that is invariant under $\theta$, i.e. $\P(\theta^{-1} A) = \P(A) = \P( \theta A)$ for every $A \in \mathcal{F}$, we call the tuple $\big(\Omega, \mathcal{F},\P,\theta\big)$ a \emph{measure-preserving dynamical system}. This system is named \emph{ergodic} if every invariant measurable set under $\theta$ has the probability $0$ or $1$.
		
		\item Assume $(\Omega,\mathcal{F},\P,\theta)$ is a measure-preserving dynamical system and $(\{E_{\omega}\}_{\omega \in \Omega},\Delta)$ a measurable field of Banach spaces. A \emph{continuous cocycle on $\{E_{\omega}\}_{\omega \in \Omega}$} consists of a family of continuous maps
		\begin{align}\label{eqn:def_cocycle}
			\varphi_{\omega} \colon E_{\omega} \to E_{\theta \omega}.
		\end{align}
		If $\varphi$ is a continuous cocycle, we set $\varphi^n_{\omega} \colon E_{\omega} \to E_{\theta^n \omega}$ where 
		\begin{align*}
			\varphi^n_{\omega} := \varphi_{\theta^{n-1}\omega} \circ \cdots \circ \varphi_{\omega}.
		\end{align*}
		We also define $\varphi^0_{\omega} := \operatorname{Id}_{E_{\omega}}$. If the maps
		\begin{align*}
			\omega \mapsto \| \varphi^n_{\omega}(g(\omega)) \|_{E_{\theta^n \omega}}, \quad n \in \N
		\end{align*}
		are measurable for every $g \in \Delta$, we say that \emph{$\varphi$ acts on $\{E_{\omega}\}_{\omega \in \Omega}$}. In this case, we will speak of a \emph{continuous random dynamical system on a field of Banach spaces}. If the map \eqref{eqn:def_cocycle} is bounded linear/compact, we call $\varphi$ a bounded linear/compact cocycle.
	\end{itemize}
	
	Next, we define stationary points for cocycles acting on measurable fields.
	\begin{definition}\label{stationary}
		Let $ \lbrace E_{\omega}\rbrace_{\omega\in \Omega} $ be a measurable field of Banach spaces and $ \varphi^{n}_{\omega} $ a nonlinear cocycle acting on it. A map $Y:\Omega\longrightarrow \prod_{\omega\in\Omega}E_{\omega}$ is called \emph{stationary} for $ \varphi^{n}_{\omega} $  provided
		\begin{itemize}
			\item[(i)]$Y_{\omega}\in E_{\omega}$,
			\item[(ii)]$\varphi^n_\omega(Y_{\omega})=Y_{\theta^{n}\omega}$ and
			\item[(iii)]$\omega\rightarrow\Vert Y_{\omega}\Vert $ is measurable.
		\end{itemize}
	\end{definition}
	Note that a stationary point for a random dynamical system can be regarded as a natural generalization of a fixed point in (deterministic) dynamical systems. If $ \varphi^{n}_{\omega}$ is Fr\'echet differentiable, one can easily check that the derivative around a stationary solution also enjoys the cocycle property, i.e for $ \psi^n_\omega(.) = D_{Y_{\omega}}\varphi^{n}_{\omega}(.) $, one has
	\begin{align*}
		\psi^{n+m}_{\omega}(.)=\psi^{n}_{\theta^{m}\omega}\big{(}\psi^m_\omega(.)\big{)}.
	\end{align*} 
	
	The key ingredient to prove the center manifold theorem is the following version of a \emph{Multiplicative Ergodic Theorem} which we call the semi-invertible Oseledets theorem on fields of Banach spaces, for a proof, cf. \cite[Theorem 4.17]{GVRS22} and \cite[Theorem 1.21]{GVR23}.
	
	\begin{theorem}\label{thm:MET_Banach_fields}
		Let $(\Omega,\mathcal{F},\mathbb{P},\theta)$ be an ergodic measure-preserving dynamical system. Assume that $\psi$ is a compact linear cocycle acting on a measurable field of Banach spaces $\{E_{\omega}\}_{\omega \in \Omega}$. For $\mu \in \R \cup \{-\infty\}$ and $\omega \in {\Omega}$, set
		\begin{align*}
			F_{\mu}(\omega) := \big{\lbrace} x\in E_{\omega}\, :\, \limsup_{n\rightarrow\infty} \frac{1}{n} \log \Vert\psi^n_{\omega}(x) \Vert \leq \mu \big{\rbrace}. 
		\end{align*} 
		Assume that
		\begin{align*}
			\log^+ \Vert\psi_{\omega} \Vert \in L^{1}(\Omega)
		\end{align*}
		and that
		\begin{align*}
			\omega \mapsto \| \psi^k_{\theta^n \omega}(\psi^n_{\omega}(g(\omega)) - \tilde{g}(\theta^n \omega)) \|_{E_{\theta^{n+k} \omega}}
		\end{align*}
		is measurable for every $g, \tilde{g} \in \Delta$ and $n,k \geq 0$. 
		
		% $\varphi(1,\omega):\pi^{-1}({\omega})\rightarrow\pi^{-1}{(\theta\omega)} $ is compact for every $ \omega\in\Omega $ also assume  $ \theta :\Omega\rightarrow\Omega$ is an ergodic random dynamical system and $ \log\big{(}\Vert\varphi(1,\omega)\Vert\big{)}\in L^{1}(\Omega) $
		Then there is a measurable $\theta$-invariant set $\tilde{\Omega} \subset \Omega$ of full measure and a decreasing sequence $\{\mu_i\}_{i \geq 1}$, $\mu_i \in [-\infty, \infty)$ (Lyapunov exponents) with the properties that $\lim_{n \to \infty} \mu_n = - \infty$ and either $\mu_i > \mu_{i+1}$ or $\mu_i = \mu_{i+1} = -\infty$ such that for every $\omega \in \tilde{\Omega}$, $F_{\mu_1}(\omega) = E_{\omega}$ and
		\begin{align}
			x\in F_{\mu_{i}}(\omega)\setminus F_{\mu_{i+1}}(\omega) \quad \text{if and only if}\quad \lim_{n\rightarrow\infty}\frac{1}{n}\log \Vert\psi^n_{\omega}(x) \| = \mu_{i}.
		\end{align}
		Moreover, there are numbers $m_1, m_2, \ldots$ such that $\operatorname{codim} F_{\mu_j}(\omega) = m_1 + \ldots + m_{j-1}$ for every $\omega \in \tilde{\Omega}$. Furthermore, for every $i \geq 1$ with $\mu_i > \mu_{i+1}$ and $\omega \in \tilde{\Omega}$, there is an $m_i$-dimensional subspace $H^i_\omega$ with the following properties:
		\begin{itemize}
			\item[(i)] (Invariance)\ \ $\psi_{\omega}^k(H^i_{\omega}) = H^i_{\theta^k \omega}$ for every $k \geq 0$.
			\item[(ii)] (Splitting)\ \ $H_{\omega}^i \oplus F_{\mu_{i+1}}(\omega) = F_{\mu_i}(\omega)$. In particular,
			\begin{align*}
				E_{\omega} = H^1_{\omega} \oplus \cdots \oplus H^i_{\omega} \oplus   F_{\mu_{i+1}}(\omega).
			\end{align*}
			\item[(iii)] ('Fast-growing' subspace)\ \ For each $ h_{\omega}\in H^{i}_{\omega}\setminus \{0\} $,
			\begin{align*}
				\lim_{n\rightarrow\infty}\frac{1}{n}\log\Vert \psi^{n}_{\omega}(h_{\omega})\Vert = \mu_{i}
			\end{align*}
			and
			\begin{align*}
				\lim_{n\rightarrow\infty}\frac{1}{n}\log\Vert (\psi^{n}_{\theta ^{-n}\omega})^{-1}(h_{\omega})\Vert=-\mu_{i}.
			\end{align*}
			\item [(iv)] ('Angle vanishing') Let $\tilde{H}^i_{\omega}$ be a subspace of $H^{i}_{\omega}$ and $h_{\omega}\in H^{i}_{\omega}\setminus \tilde{H}^{i}_{\omega}$. Then
			\begin{align*}
				\lim_{n\rightarrow\infty}\frac{1}{n}\log d(\psi^{n}_{\omega}(h_{\omega}),\psi^{n}_{\omega}(\tilde{H}^{i}_{\omega})) = \mu_{i}
			\end{align*}
			and
			\begin{align*}
				\lim_{n\rightarrow\infty}\frac{1}{n}\log d\big((\psi^{n}_{\theta ^{-n}\omega})^{-1}(h_{\omega}),(\psi^{n}_{\theta ^{-n}\omega})^{-1}(\tilde{H}^{i}_{\omega}) \big) = -\mu_{i}.
			\end{align*}
		\end{itemize} 
	\end{theorem}
	\begin{remark}
		\begin{itemize}
			\item[(i)] The angle vanishing property in item (iv) appears in \cite [Lemma 1.20]{GVR23}.
			\item[(ii)] Note that our cocycle is not necessarily injective. However, for every $i\geq 1$ and $\omega\in \tilde{\Omega}$, due to the invariance property (item (i)), we can define the inverse of the cocycle restricted to any finite direct sum of the spaces $H^{i}_{\omega}$. In addition, as shown in \cite[Theorem 1.21]{GVR23}, the properties listed in Theorem 1.2 regarding the fast-growing subspaces uniquely determine these subspaces. This leads to a canonical way to define the inverse function up to any finite index.
		\end{itemize}
		
	\end{remark}

	\section{Main part}\label{main}
	
	%
	%\\
	
	%[\psi^{-n}_{\theta^n\omega}]^{-1}(v_\omega)
	%\\ \quad -\sum_{0\leq j\leq -n-1}\big{[}[\psi^{1+j}_{\theta^n\omega}]^{-1}\circ\Pi_{C_{\theta^{1+j+n}\omega}\parallel S_{\theta^{1+j+n}\omega}\oplus U_{\theta^{1+j+n}\omega}}\big{]}P_{\theta^{j+n}\omega}(\Pi^{{j+n}}_{\omega}[\Gamma])\\
	%\quad+\sum_{j\geq -n}[\psi^{j+n}_{\theta^{-j}\omega}\circ\Pi_{S_{\theta^{-j}\omega}\parallel U_{\theta^{-j}\omega}\oplus C_{\theta^{-j}\omega}}]P_{\theta^{-1-j}\omega}(\Pi^{-1-j}_{\omega}[\Gamma])
	%\\
	%\qquad -\sum_{j<-n+1}\big{[}[\psi^{-n+1-j}_{\theta^{n}\omega}]^{-1}\circ\Pi_{U_{\theta^{1-j\omega}}\parallel C_{\theta^{1-j}\omega}\oplus S_{\theta^{1-j}\omega}}\big{]}P_{\theta^{-j}\omega}(\Pi^{-j}_{\omega}[\Gamma]), &\text{for } n<0 .
	%
	We aim to prove a general center manifold theorem for a nonlinear cocycle acting on a measurable field of Banach spaces. For the rest of the paper, we will assume that $\big(\Omega, \mathcal{F},\P,\theta\big)$  is a {measure-preserving dynamical system} such that $\theta$ is ergodic and invertible. We also assume that $ \lbrace E_{\omega}\rbrace_{\omega\in \Omega} $ is a measurable field of Banach spaces and $ \varphi^{n}_{\omega} $ is a nonlinear cocycle acting on it. We will assume
	\begin{assumption}\label{ASSU}
		$ \varphi^{n}_{\omega}$ is  Fr\'echet differentiable and $Y \colon \Omega\longrightarrow \prod_{\omega\in\Omega}E_{\omega}$ is a stationary point for it such that:
		\begin{itemize}
			\item For $ \psi^n_\omega(.) \coloneqq D_{Y_{\omega}}\varphi^{n}_{\omega}(.)$, 
			\begin{align*}
				\log^+ \Vert\psi_{\omega} \Vert \in L^{1}(\Omega)
			\end{align*}
			where $\log^+ x \coloneqq \max \{0,\log x \}$.
			\item We assume that for
			\begin{align} \label{linearization}
				\begin{split}
					P_{\omega} \colon E_{\omega} &\to E_{\theta\omega }\\
					\xi_{\omega} &\mapsto \varphi^{1}_{\omega} (Y_{\omega}+\xi_{\omega})-\varphi^{1}_{\omega}(Y_{\omega})-\psi^{1}_{\omega}(\xi_{\omega}), 
				\end{split}
			\end{align}
			there exists a random variable $R \colon \Omega \to [0,\infty)$ with the property that
			\begin{align*}
				\liminf_{n\rightarrow \infty}\frac{1}{n}\log R(\theta^{n}\omega) \geq 0
			\end{align*}
			almost surely and that for $ \Vert \xi_{\omega}\Vert , \Vert\tilde{\xi}_{\omega}\Vert <R(\omega) $, one has
			\begin{align}\label{eqn:diff_bound_P}
				\Vert P_{\omega} (\xi_{\omega}) - P_{\omega} (\tilde{\xi}_{\omega})\Vert\leq \Vert\xi_{\omega}-\tilde{\xi}_{\omega}\Vert f(\theta\omega)  h{(}\Vert\xi_{\omega}\Vert +\Vert\tilde{\xi}_{\omega}\Vert{)}
			\end{align}
			almost surely where $f \colon \Omega\rightarrow \mathbb{R}^{+} $ is a measurable function with the property that 
			\begin{align*}
				\lim_{n\rightarrow\infty}\frac{1}{n}\log^{+}f(\theta^{n}\omega) = 0  
			\end{align*}
			%	and $ h(x)=x^{r}g(x)$ for some $r > 0$ where $ g:\mathbb{R}\rightarrow\mathbb{R}^{+} $ is an increasing $ C^{1} $ function.
			almost surely. Furthermore, $h$ is assumed to be a nonnegative and increasing function such that  $h(0)=0$ and for some $r>0$, $\limsup_{x\rightarrow 0}\frac{h(x)}{|x|^r}<\infty$.
		\end{itemize}
	\end{assumption}
	Note these assumptions are enabling us to apply Theorem \ref{thm:MET_Banach_fields} to the linearized cocycle around the stationary point $\psi^{n}_{\omega}(.)$. Let us assume $\mu_1>0$ and that we have a zero Lyapunov exponent. Then for $\mu^{-} \coloneqq \max \lbrace \mu_{i} \colon  \mu_{i}<0\rbrace$, we define
	\begin{align}\label{RTR}
		S_{\omega} \coloneqq F_{\mu^{-}}(\omega),\ \ \ U_{\omega} \coloneqq \bigoplus_{i: \mu_{i}>0}H_{\omega}^{i}, \ \ \ \text{and} \ \ C_{\omega} \coloneqq H^{i_{c}}_{\omega},
	\end{align}
	where $\mu_{i_c}=0$. We also set $\mu^{+} \coloneqq \min \lbrace \mu_{i}:  \mu_{i}>0\rbrace$. 
	Note that from \cite[Lemma 1.18]{GVR23}, for every $\omega \in \tilde{\Omega}$:
	\begin{align}\label{pro_invv}
		\begin{split}
			\lim_{j\rightarrow \pm \infty}\frac{1}{n}\log [\Vert \Pi_{ C_{\theta^{j}\omega}\parallel S_{\theta^{j}\omega}\oplus U_{\theta^{j}\omega}}\Vert ]= &
			\lim_{n\rightarrow \pm \infty}\frac{1}{n}\log [\Vert \Pi_{ U_{\theta^{j}\omega}\parallel S_{\theta^{j}\omega}\oplus C_{\theta^{j}\omega}}\Vert ]= \\&
			\lim_{n\rightarrow \pm \infty}\frac{1}{n}\log [\Vert \Pi_{ S_{\theta^{j}\omega}\parallel U_{\theta^{j}\omega}\oplus C_{\theta^{j}\omega}}\Vert ]=0.    
		\end{split}
	\end{align}

	\begin{definition}
		Let $\delta :\mathbb{R}\rightarrow [0,1]$ be a smooth function such that $ \text{supp}(\delta)\subset [-2,2]$, $\delta|_{[-1,1]}=1$ and $\sup_{x\in\mathbb{R}}|\delta^{\prime}(x)|\leq 2$. Also, assume $\rho$ to be a positive random variable. Then we set 
		\begin{align*}
			P_{\omega,\rho}(\xi_{\omega}) \coloneqq \delta(\frac{\Vert\xi_\omega\Vert}{\rho(\theta\omega)})P_{\omega}(\xi_\omega).
		\end{align*}
	\end{definition}
	
	Note that, from \eqref{eqn:diff_bound_P}  and our assumptions,
	\begin{align}\label{LIPs}
		\Vert P_{\omega,\rho}(\xi_{\omega}) - P_{\omega,\rho}(\tilde{\xi}_{\omega})\Vert\leq 5\Vert\xi_{\omega}-\tilde{\xi}_{\omega}\Vert f(\theta\omega)h\big{(}4\rho(\theta\omega)\big{)}.
	\end{align}

	Before proving the main result, we need some auxiliary lemmas.
	\begin{lemma}
		Choose $\tilde{\Omega}$ as in Theorem \ref{thm:MET_Banach_fields}. Let $\omega \in \tilde{\Omega}$ and for $p \leq q$, define ${H}^{p,q}_{\omega} \coloneqq \bigoplus_{p \leq i \leq q} H^{i}_{\omega}$. Assume $\text{dim}({H}^{p,q}_{\omega}) = k$ and let $\{\xi^{t}_{\omega}\}_{1 \leq t \leq k}$ be a basis for ${H}^{p,q}_{\omega}$, i.e. ${H}^{p,q}_{\omega} = \langle \xi^{t}_{\omega} \rangle_{1 \leq t \leq k}$. Then for $m,n\in \mathbb{Z}$, we have
		\begin{align}
			\Vert\psi^{n}_{\theta^m\omega}|_{{H}^{p,q}_{\theta^{m}\omega}}\Vert\leq\sum_{1\leq t\leq k}\frac{\Vert\psi^{m+n}_{\omega}(\xi^{t}_{\omega})\Vert}{d\big{(}\psi_{\omega}^{m}(\xi^{t}_{\omega}), \langle \psi_{\omega}^{m}(\xi^{t^{\prime}}_{\omega}) \rangle_{1\leq t^{\prime}\leq k, t^{\prime}\neq t}\big{)}}.
		\end{align}
		
		\begin{proof}
			Remember for $n<0$, $\psi^{n}_{\omega}=[\psi^{-n}_{\theta^{n}\omega}]^{-1}|_{H_{\omega}^1}$. We need to consider several cases by distinguishing the sign of $m,n, m+n$. Here we take the case $n<0, m+n>0$. Other cases can be proved similarly and are even easier. Assume $\xi=\sum_{1\leq t\leq k}c_t\frac{\psi^{m}_{\omega}(\xi_t)}{\Vert\psi^{m}_{\omega}(\xi_t)\Vert}$. Clearly, we have
			\begin{align}\label{coeficicents_approxiamtion}
				\frac{\Vert\xi\Vert \Vert\psi^{m}_{\omega}(\xi_t)\Vert}{\vert c_t\vert}\geq d\big{(}\psi_{\omega}^{m}(\xi^{t}_{\omega}), \langle \psi_{\omega}^{m}(\xi^{t^{\prime}}_{\omega}) \rangle_{1\leq t^{\prime}\leq k, t^{\prime}\neq t}\big{)}.
			\end{align} 
			For $1\leq t\leq k$, by definition,
			\begin{align}\label{co-}
				\begin{split}
					\psi^{n}_{\theta^{m}}&\big{(}\psi^{m}_{\omega}(\xi_t)\big{)}=[\psi^{-n}_{\theta^{m+n}\omega}]^{-1}\big{(}\psi_{\omega}^{m}(\xi_t)\big{)}\\ &\quad=[\psi^{-n}_{\theta^{m+n}\omega}]^{-1}\big{(}\psi^{-n}_{\theta^{m+n}\omega}(\psi^{m+n}_{\omega}(\xi_t))\big{)}=\psi_{\omega}^{m+n}(\xi_t).
				\end{split}
			\end{align}
			Accordingly from \eqref{coeficicents_approxiamtion} and \eqref{co-}
			\begin{align}\label{Operator _ Estimate}
				\begin{split}
					\Vert\psi_{\theta^m\omega}^{n}(\xi)\Vert&\leq \sum_{1\leq t\leq k} \vert c_t\vert\frac{\Vert\psi_{\omega}^{m+n}(\xi_t)\Vert}{\Vert\psi_{\omega}^{m}(\xi_t)\Vert}\\
					&\leq\sum_{1\leq t\leq k}\frac{\Vert\psi^{m+n}_{\omega}(\xi^{t}_{\omega})\Vert}{d\big{(}\psi_{\omega}^{m}(\xi^{t}_{\omega}), \langle \psi_{\omega}^{m}(\xi^{t^{\prime}}_{\omega}) \rangle_{1\leq t^{\prime}\leq k, t^{\prime}\neq t}\big{)}}\Vert\xi\Vert,
				\end{split}
			\end{align}
			which proves our claim.
		\end{proof}	 
		%	\begin{align}\label{inverse-bound}}
		%	\Vert[\psi^{n}_{\theta^{m}\omega}]^{-1}\Vert_{L[{{H}^{p,q}_{\theta^{m+n}\omega}},{H}^{p,q}_{\theta^{m}\omega}]}\leq\sum_{1\leq t\leq k}\frac{\Vert\psi^{m}_{\omega}(\xi^{t}_{\omega})\Vert}{d\big{(}\psi_{\omega}^{m+n}(\xi^{t}_{\omega}), \langle \psi_{\omega}^{m+n}(\xi^{t^{\prime}}_{\omega}) \rangle_{t^{\prime}\neq t}\big{)}}
		%	\end{align}
	%	and 
	%	\begin{align}\label{inn_BBB}
		%	\Vert [\psi^{p}_{\sigma^{n}\omega}]^{-1}\Vert_{L[{U_{\sigma^{n-p}\omega}},U_{\sigma^{n}\omega}]}\leq \sum_{1\leq t\leq k}\frac{\Vert (\psi^{n}_{\sigma^{n}\omega})^{-1}(\xi^{t}_{\omega})\Vert}{ \Vert (\psi^{n-p}_{\sigma^{n-p}\omega})^{-1}(\xi^{t}_{\omega})\Vert}\times\frac{ \Vert (\psi^{n-p}_{\sigma^{n-p}\omega})^{-1}(\xi^{t}_{\omega})\Vert}{d\big{(}(\psi^{n-p}_{\sigma^{n-p}\omega})^{-1}(\xi^{t}_{\omega}),\langle (\psi^{n-p}_{\sigma^{n-p}(\omega)})^{-1}(\xi^{t^{\prime}}_{\omega})\rangle_{t^{\prime}\neq t}\big{)}}.
		%	\end{align}
\end{lemma}

Recall the definition of \(S_\omega\), $U_\omega$ and \( C_\omega \) in \eqref{RTR}.

\begin{lemma}\label{lemma:prop_F}
	For $\omega\in \tilde{\Omega}$ and $\epsilon>0$, set
	\begin{align*}
		F^{S}_{\epsilon}(\omega):=\sup_{n\geq 0}\Vert\psi^{n}_{\omega}|_{S_{\omega}}\Vert\exp(-n(\mu^{-}+\epsilon)), \ \ \ F^{U}_{\epsilon}(\omega):=\sup_{n\leq 0}\Vert\psi^{n}_{\omega}|_{U_{\omega}}\Vert\exp(n(-\mu^{+}+\epsilon)).
	\end{align*}
	Then
	\begin{align*}
		\lim_{m\rightarrow \infty}\frac{1}{m}\log^{+}[F^{S}_{\epsilon}(\theta^{m}\omega)]=\lim_{m\rightarrow-\infty}\frac{1}{m}\log^{+}[F^{U}_{\epsilon}(\theta^{m})]=0.
	\end{align*}
	Similarly for
	\begin{align*}
		F^{C,1}_{\epsilon}(\omega):=\sup_{n\geq 0}\Vert\psi^{n}_{\omega}|_{C_{\omega}}\Vert\exp(-n\epsilon), \ \ \ F^{C,-1}_{\epsilon}(\omega):=\sup_{n\leq 0}\Vert\psi_{\omega}^{n}|_{C_\omega}\Vert\exp(n\epsilon),
	\end{align*}
	we have
	\begin{align*}
		\lim_{n\rightarrow \infty}\frac{1}{m}\log^{+}[F^{C,1}_{\epsilon}(\theta^{m}\omega)]=\lim_{m\rightarrow-\infty}\frac{1}{m}\log^{+}[F^{C,-1}_{\epsilon}(\theta^{m}\omega)]=0.
	\end{align*}
\end{lemma}
\begin{proof}
	For $F^{S}_{\varepsilon}$, the statement is shown in \cite[Lemma 2.3]{GVR23}. We only prove the claim for $F^{C,-1}_{\varepsilon}$, the remaining assertions can be shown in the same way. Assume that \( C_{\omega} = \langle \xi^{t}_{\omega} \rangle_{1 \leq t \leq k} \). Applying \eqref{Operator _ Estimate} yields
	\begin{align*}
		F^{C,-1}(\theta^{m}\omega)&\leq\sup_{n\leq 0}\Vert\psi_{\omega}^{n}(\theta^{m}\omega)\Vert\exp(n\epsilon)\\ \quad &\leq\sup_{n\leq 0}\big{[}\sum_{1\leq t\leq k}\frac{\Vert\psi^{m+n}_{\omega}(\xi^{t}_{\omega})\Vert}{d\big{(}\psi_{\omega}^{m}(\xi^{t}_{\omega}), \langle \psi_{\omega}^{m}(\xi^{t^{\prime}}_{\omega}) \rangle_{1\leq t^{\prime}\leq k, t^{\prime}\neq t}\big{)}}\exp(n\epsilon)\big{]}.
	\end{align*}
	Using (iv) in Theorem \ref{thm:MET_Banach_fields} proves the claim.
\end{proof}
\begin{proposition}\label{CENTER}
	For $\nu>0$, let ${\Sigma}_\omega \coloneqq \prod_{j\in\mathbb{Z}}E_{\theta^{j}\omega}$ and
	\begin{align*}
		\Sigma_{\omega}^{\nu} \coloneqq \bigg{\lbrace}\Gamma\in{\Sigma}_\omega: \ \Vert\Gamma\Vert=\sup_{j\in\mathbb{Z}}\big{[}\Vert\Pi^{j}_{\omega}[\Gamma]\Vert\exp(-\nu \vert j\vert)\big{]}<\infty \bigg{\rbrace}.
	\end{align*}
	Assume that $0 < \nu < \min\{\mu^+, - \mu^-\}$. Then there is a random variable $\rho \colon \Omega \to (0,\infty)$ such that the following map is well defined:
	\begin{align*}
		&I_{_{\omega}} \colon C_{\omega}\times\Sigma_{\omega}^{\nu} \cap B(0,R(\omega))\rightarrow\Sigma_{\omega}^{\nu}, \\
		&{\Pi}^{n}_{\omega}\big{[}I_{\omega}(v_{\omega} ,\Gamma)\big{]} = \begin{cases}
			\psi^{n}_{\omega}(v_{\omega}) + C(n,\omega,\Gamma)  \\
			\quad - \sum_{j\geq n+1}\big{[}\psi^{n-j}_{\theta^{j}\omega}\circ\Pi_{U_{\theta^{j}\omega}\parallel C_{\theta^{j}\omega}\oplus S_{\theta^{j}\omega}}\big{]}  P_{\theta^{j-1}\omega,\rho}\big{(}\Pi^{j-1}_{\omega}[\Gamma]\big{)}\\
			\qquad +\sum_{j\leq n}[\psi^{n-j}_{\theta^{j}\omega}\circ\Pi_{S_{\theta^{j}\omega}\parallel U_{\theta^{j}\omega}\oplus C_{\theta^{j}\omega} }]P_{\theta^{j-1}\omega,\rho}(\Pi^{j-1}_{\omega}[\Gamma])  &\text{for } n \ne 0,\\
			\\
			v_{\omega}-\sum_{j \geq 1}\big{[}\psi^{-j}_{\theta^{j}\omega}\Pi_{U_{\theta^{j}\omega}\parallel C_{\theta^{j}\omega}\oplus S_{\theta^{j}\omega}}\big{]}P_{\theta^{j-1}\omega,\rho}(\Pi^{j-1}_{\omega}[\Gamma])\\ \quad +\sum_{j\leq 0}[\psi^{-j}_{\theta^{j}\omega}\circ\Pi_{S_{\theta^{j}\omega}\parallel U_{\theta^{j}\omega}\oplus C_{\theta^{j}\omega}}]P_{\theta^{j-1}\omega,\rho}(\Pi^{j-1}_{\omega}[\Gamma]) &\text{for } n = 0,
		\end{cases}
	\end{align*}
	where
	\begin{align*}
		C(n,\omega,\Gamma)=\begin{cases}
			\sum_{1\leq j\leq n}\big{[}\psi^{n-j}_{\theta^{j}\omega}\circ\Pi_{ C_{\theta^{j}\omega}\parallel S_{\theta^{j}\omega}\oplus U_{\theta^{j}\omega}}\big{]}P_{\theta^{j-1}\omega,\rho}\big{(}\Pi^{j-1}_{\omega}[\Gamma]\big{)}\
			&\text{for } n>0,\\
			\\
			-\sum_{n\leq j\leq -1}\big{[}\psi^{n-j}_{\theta^{j}\omega}\circ\Pi_{ C_{\theta^{j}\omega}\parallel S_{\theta^{j}\omega}\oplus U_{\theta^{j}\omega}}\big{]}P_{\theta^{j-1}\omega,\rho}\big{(}\Pi^{j-1}_{\omega}[\Gamma]\big{)}\
			&\text{for } n<0,
		\end{cases}
	\end{align*}
	In addition, for every $v_{\omega} \in C_\omega$, the equation $I_{\omega}(v_{\omega},\Gamma)=\Gamma$ admits a unique solution.
	%and for $n<0$, $\psi^{n}_{\omega}|_{{H}^{p,q}_\omega} :=[\psi^{-n}_{\theta^n\omega}]^{-1}|_{{H}^{p,q}_\omega}.$
\end{proposition}

\begin{proof}
	By definition
	\begin{align*}
		\big{\Vert}\Pi^{n}_{\omega}&[I_{\omega}(v_{\omega},\Gamma)]\big{\Vert}\leq\Vert\psi^{n}_{\omega}|_{C_\omega}\Vert \ \Vert v_{\omega}\Vert\\
		&\quad +\sum_{\substack{n\leq j\leq -1 \ \text{or}\\ 1\leq j\leq n}}\Vert\psi^{n-j}_{\theta^{j}\omega}|_{C_{\theta^{j}\omega}}\Vert \ \Vert\Pi_{C_{\theta^{j}\omega}\parallel S_{\theta^{j}\omega}\oplus U_{\theta^{j}\omega}}\Vert \ \Vert\Pi^{j-1}_{\omega}[\Gamma]\Vert f(\theta^{j}\omega)h\big(\Vert\Pi^{j-1}_{\omega}[\Gamma]\Vert\big)\delta \biggl(\frac{\Vert\Pi^{j-1}_{\omega}[\Gamma]\Vert}{\rho(\theta^{j}\omega)} \biggr)
		\\
		&\qquad+\sum_{j\geq n+1} \Vert\psi^{n-j}_{\theta^{j}\omega}|_{U_{\theta^{j}\omega}}\Vert \ \Vert\Pi_{U_{\theta^{j}\omega}\parallel C_{\theta^{j}\omega}\oplus S_{\theta^{j}\omega}}\Vert \ \Vert\Pi^{j-1}_{\omega}[\Gamma]\Vert f(\theta^{j}\omega)h\big(\Vert\Pi^{j-1}_{\omega}[\Gamma]\Vert\big)\delta  \biggl( \frac{\Vert\Pi^{j-1}_{\omega}[\Gamma]\Vert}{\rho(\theta^{j}\omega)} \biggr) \\
		&\ \qquad +\sum_{j\leq n}\Vert\psi^{n-j}_{\theta^{j}\omega}|_{S_{\theta^{j}\omega}}\Vert \ \Vert\Pi_{S_{\theta^{j}\omega}\parallel U_{\theta^{j}\omega}\oplus C_{\theta^{j}\omega}}\Vert \ \Vert\Pi^{j-1}_{\omega}[\Gamma]\Vert f(\theta^{j}\omega)h\big(\Vert\Pi^{j-1}_{\omega}[\Gamma]\Vert\big)\delta \biggl( \frac{\Vert\Pi^{j-1}_{\omega}[\Gamma]\Vert}{\rho(\theta^{j}\omega)} \biggr).
	\end{align*}
	Choose $0 < \varepsilon < \nu$. Remember that $\Vert\Pi^{j}_{\omega}[\Gamma]\Vert\leq \exp(\nu|j|)\Vert\Gamma\Vert$. Consequently,
	\begin{align*}
		&\exp(-\nu|n|)\big{\vert}\Pi^{n}_{\omega}[I_{\omega}(v_{\omega},\Gamma)]\big{\vert} \leq \exp((\epsilon-\nu)|n|)F^{C,sgn(n)}_{\epsilon}(\omega)\Vert v_{\omega}\Vert\\
		&\qquad + \Vert\Gamma\Vert\sum_{\substack{n\leq j\leq -1 \ \text{or}\\ 1\leq j\leq n}}\exp((\epsilon-\nu)|n-j|+\nu)F^{C,sgn(n)}_{\epsilon}(\theta^{j}\omega)\Vert\Pi_{C_{\theta^{j}\omega}\parallel S_{\theta^{j}\omega}\oplus U_{\theta^{j}\omega}}\Vert f(\theta^{j}\omega) h\big(2\rho(\theta^{j}\omega)\big)\\
		&\qquad + \Vert\Gamma\Vert\sum_{j\geq n+1}\exp((j-n)(-\mu^++\epsilon+\nu)+\nu)F^{U}_{\epsilon}(\theta^{j}\omega)\Vert\Pi_{U_{\theta^{j}\omega}\parallel C_{\theta^{j}\omega}\oplus S_{\theta^{j}\omega}}\Vert f(\theta^{j}\omega)h(2\rho(\theta^{j}\omega))\\
		&\qquad + \Vert\Gamma\Vert\sum_{j\leq n}\exp((n-j)(\mu^-+\epsilon+\nu)+\nu)F^{S}_{\epsilon}(\theta^{j}\omega)\Vert\Vert\Pi_{S_{\theta^{j}\omega}\parallel U_{\theta^{j}\omega}\oplus C_{\theta^{j}\omega}}\Vert f(\theta^{j}\omega)h\big(2\rho(\theta^{j}\omega)\big).
	\end{align*}
	Similarly, by \eqref{LIPs},
	\begin{align*}
		&\exp(-\nu|n|)\big{\Vert}\Pi^{n}_{\omega}[I_{\omega}(v_{\omega},\Gamma)]-\Pi^{n}_{\omega}[I_{\omega}(v_{\omega},\tilde{\Gamma})]\big{\Vert}\\ 
		\leq\ &\Vert\Gamma-\tilde{\Gamma}\Vert\sum_{\substack{n\leq j\leq -1 \ \text{or}\\ 1\leq j\leq n}}5\exp((\epsilon-\nu)|n-j|+\nu)F^{C,sgn(n)}_{\epsilon}(\theta^{j}\omega)\Vert\Pi_{C_{\theta^{j}\omega}\parallel S_{\theta^{j}\omega}\oplus U_{\theta^{j}\omega}}\Vert\, f(\theta^{j}\omega)h(4\rho(\theta^{j}\omega))\\
		&\ + \Vert\Gamma-\tilde{\Gamma}\Vert\sum_{j\geq n+1}5\exp((j-n)(-\mu^++\epsilon+\nu)+\nu)F^{U}_{\epsilon}(\theta^{j}\omega)\Vert\Pi_{U_{\theta^{j}\omega}\parallel C_{\theta^{j}\omega}\oplus S_{\theta^{j}\omega}}\Vert\, f(\theta^{j}\omega)h(4\rho(\theta^{j}\omega))\\
		&\ + \Vert\Gamma-\tilde{\Gamma}\Vert\sum_{j\leq n}5\exp((n-j)(\mu^-+\epsilon+\nu)+\nu)F^{S}_{\epsilon}(\theta^{j}\omega)\Vert\Vert\Pi_{S_{\theta^{j}\omega}\parallel U_{\theta^{j}\omega}\oplus C_{\theta^{j}\omega}}\Vert\, f(\theta^{j}\omega)h(4\rho(\theta^{j}\omega)).
	\end{align*}
	For some given $M_{\epsilon}>0$, let  
	\begin{align}\label{eqn:T_omega}
		T(\omega) \coloneqq \frac{M_{\epsilon}}{f(\omega)} \min \biggl\{ \frac{1}{F_{\epsilon}^{C,1}(\omega) \Vert\Pi_{C_{\omega}\parallel S_{\omega}\oplus U_{\omega}}\Vert},&\frac{1}{F_{\epsilon}^{C,-1}(\omega) \Vert\Pi_{C_{\omega}\parallel S_{\omega}\oplus U_{\omega}}\Vert f(\omega)} ,\\& \frac{1}{F_{\epsilon}^{U}(\omega) \Vert\Pi_{U_{\omega}\parallel C_{\omega}\oplus S_{\omega}}\Vert}, \frac{1}{F_{\epsilon}^{S}(\omega) \Vert\Pi_{S_{\omega}\parallel U_{\omega}\oplus C_{\omega}}\Vert} \biggr\}
	\end{align}
	and
	\begin{align}\label{eqn:def_rho}
		\rho(\omega) \coloneqq \frac{1}{4} h^{-1}(T(\omega)).
	\end{align}
	Choosing $\varepsilon > 0$ sufficiently small, the recent inequalities yield
	\begin{align*}
		&\Vert I_{\omega}(v_{\omega} ,\Gamma)\Vert\leq \max \lbrace F^{C,-1}_{\epsilon}(\omega),F^{C,1}_{\epsilon}(\omega) \rbrace  \Vert v_{\omega}\Vert\\
		&\quad + \Vert \Gamma\Vert M_{\epsilon}\exp(\nu) \left(\frac{1}{1-\exp(\epsilon-\nu)}+\frac{\exp(-\mu^++\epsilon+\nu)}{1-\exp(-\mu^++\epsilon+\nu)}+\frac{1}{1-\exp(\mu^-+\epsilon+\nu)}\right)
	\end{align*}
	and
	\begin{align*}
		&\Vert I_{\omega}(v_{\omega} ,\Gamma) - I_{\omega}(v_{\omega} ,\tilde\Gamma)\Vert \\
		\leq\ &\|\Gamma - \Gamma \| 5M_{\epsilon}\exp(\nu) \left( \frac{1}{1-\exp(\epsilon-\nu)}+\frac{\exp(-\mu^++\epsilon+\nu)}{1-\exp(-\mu^++\epsilon+\nu)}+\frac{1}{1-\exp(\mu^-+\epsilon+\nu)} \right).
	\end{align*}
	Set 
	\begin{align*}
		L_{\epsilon} \coloneqq M_{\epsilon}\exp(\nu) \left( \frac{1}{1-\exp(\epsilon-\nu)}+\frac{\exp(-\mu^++\epsilon+\nu)}{1-\exp(-\mu^++\epsilon+\nu)}+\frac{1}{1-\exp(\mu^-+\epsilon+\nu)} \right).
	\end{align*}  
	Choosing $M_{\varepsilon} > 0$ sufficiently small, we can obtain that $5L_{\epsilon}<1$. With this choice, the map $I_{\omega}$ is indeed well defined and $I_{\omega}(v_{\omega},\cdot)$ is a contraction. Therefore, for every $v_{\omega}\in C_{\omega}$, the equation $I_{\omega}(v_{\omega},\Gamma_{v_\omega})=\Gamma_{v_\omega}$ has a unique solution with the property that
	\begin{align*}
		\| \Gamma_{v_\omega} \|\leq \frac{\max \lbrace F^{C,-1}_{\epsilon}(\omega),F^{C,1}_{\epsilon}(\omega) \rbrace \Vert v_{\omega}\Vert}{1-L_{\epsilon}}.
	\end{align*}
\end{proof}

We will need some further properties of the fixed points $\Gamma_{v_{\omega}} = I_{\omega}(v_{\omega},\Gamma_{v_{\omega}})$ which we will state now.
\begin{lemma}\label{lemma:I_lipschitz}
	Let $I_{\omega}$ be as in Proposition \ref{CENTER}. Then
	\begin{align*}
		\|\Gamma - \tilde{\Gamma} \| \leq 2 \max\{ F^{C,1}_{\nu}(\omega),  F^{C,-1}_{\nu}(\omega) \} \|v_{\omega} - \tilde{v}_{\omega} \|
	\end{align*}
	for every $v_{\omega}, \tilde{v}_{\omega} \in C_{\omega}$ where $\Gamma = \Gamma_{v_{\omega}}$ and $\tilde{\Gamma} = \Gamma_{\tilde{v}_{\omega}}$ denote the corresponding fixed points.
	
\end{lemma}

\begin{proof}
	From the triangle inequality,
	\begin{align*}
		\|I_{\omega}(v_{\omega},\Gamma) - I_{\omega}(\tilde{v}_{\omega},\tilde{\Gamma}) \| \leq \|I_{\omega}(v_{\omega},\Gamma) - I_{\omega}(\tilde{v}_{\omega},{\Gamma}) \| + \|I_{\omega}(\tilde{v}_{\omega},\Gamma) - I_{\omega}(\tilde{v}_{\omega},\tilde{\Gamma}) \|.
	\end{align*}
	As we have seen in the proof of Proposition \ref{CENTER},
	\begin{align*}
		\|I_{\omega}(\tilde{v}_{\omega},\Gamma) - I_{\omega}(\tilde{v}_{\omega},\tilde{\Gamma}) \| \leq 5 L_{\varepsilon} \|\Gamma - \tilde{\Gamma} \| = 5 L_{\varepsilon} \|I_{\omega}(v_{\omega},\Gamma)- I_{\omega}(\tilde{v}_{\omega},\tilde{\Gamma}) \|,
	\end{align*}
	thus
	\begin{align*}
		\|I_{\omega}(v_{\omega},\Gamma) - I_{\omega}(\tilde{v}_{\omega},\tilde{\Gamma}) \| \leq \frac{1}{1 - 5 L_{\varepsilon}} \|I_{\omega}(v_{\omega},\Gamma) - I_{\omega}(\tilde{v}_{\omega},{\Gamma}) \|.
	\end{align*}
	By definition of the map $I_{\omega}$,
	\begin{align*}
		\|I_{\omega}(v_{\omega},\Gamma) - I_{\omega}(\tilde{v}_{\omega},{\Gamma}) \| \leq \max\{ F^{C,1}_{\nu}(\omega),  F^{C,-1}_{\nu}(\omega) \} \|v_{\omega} - \tilde{v}_{\omega} \|.
	\end{align*}
	Choosing $5 L_{\varepsilon} \leq \frac{1}{2}$ in the proof of Proposition \ref{CENTER} yields the claim.
\end{proof}

\begin{lemma}\label{lemma:injecitvity}
	With the same notation as in Lemma \ref{lemma:I_lipschitz},
	\begin{align*}
		\Vert \Pi^{0}_{\omega}(\Gamma)-\Pi^{0}_{\omega}(\tilde{\Gamma})\Vert\geq \Vert v_{\omega} - \tilde{v}_{\omega}\Vert - \frac{1}{4 \max\{ F^{C,1}_{\nu}(\omega),  F^{C,-1}_{\nu}(\omega) \}} \Vert\Gamma-\tilde{\Gamma}\Vert.
	\end{align*}
	In particular,
	\begin{align*}
		\Vert v_{\omega} - \tilde{v}_{\omega}\Vert \leq 2 \Vert \Pi^{0}_{\omega}(\Gamma)-\Pi^{0}_{\omega}(\tilde{\Gamma})\Vert.
	\end{align*}
\end{lemma}

\begin{proof}
	By definition,
	\begin{align*}
		&\Pi^{0}_{\omega}(\Gamma)-\Pi^{0}_{\omega}(\tilde{\Gamma}) \\
		=\ &v_{\omega} - \tilde{v}_{\omega} - \sum_{j \geq 1}\big{[}\psi^{-j}_{\theta^{j}\omega}\Pi_{U_{\theta^{j}\omega}\parallel C_{\theta^{j}\omega}\oplus S_{\theta^{j}\omega}}\big{]}\big(P_{\theta^{j-1}\omega,\rho}(\Pi^{j-1}_{\omega}[\Gamma])-(P_{\theta^{j-1}\omega,\rho}(\Pi^{j-1}_{\omega}[\tilde{\Gamma}])\big) \\
		&\quad + \sum_{j\leq 0}[\psi^{-j}_{\theta^{j}\omega}\circ\Pi_{S_{\theta^{j}\omega}\parallel U_{\theta^{j}\omega}\oplus C_{\theta^{j}\omega}}]\big(P_{\theta^{j-1}\omega,\rho}(\Pi^{j-1}_{\omega}[\Gamma])-P_{\theta^{j-1}\omega,\rho}(\Pi^{j-1}_{\omega}[\tilde{\Gamma}])\big)
	\end{align*}
	Therefore, choosing $\varepsilon$ as in the proof of Proposition \ref{CENTER},
	\begin{align*}
		&\Vert \Pi^{0}_{\omega}(\Gamma) - \Pi^{0}_{\omega}(\tilde{\Gamma}) \Vert \\
		\geq\ &\Vert v_{\omega}-\tilde{v}_{\omega}\Vert - \Vert\Gamma-\tilde{\Gamma}\Vert\sum_{j\geq 1}5\exp((j-1)(-\mu^++\epsilon+\nu)+\nu)F^{U}_{\epsilon}(\theta^{j}\omega)\Vert\Pi_{U_{\theta^{j}\omega}\parallel C_{\theta^{j}\omega}\oplus S_{\theta^{j}\omega}}\Vert f(\theta^{j}\omega)h(4\rho(\theta^{j}\omega))\\
		&\quad - \Vert\Gamma-\tilde{\Gamma}\Vert\sum_{j\leq 0}5\exp((-j)(\mu^-+\epsilon+\nu)+\nu)F^{S}_{\epsilon}(\theta^{j}\omega)\Vert\Vert\Pi_{S_{\theta^{j}\omega}\parallel U_{\theta^{j}\omega}\oplus C_{\theta^{j}\omega}}\Vert f(\theta^{j}\omega)h(4\rho(\theta^{j}\omega)).
	\end{align*}
	For given $\tilde{M}_{\epsilon}>0$, define 
	\begin{align*}
		\tilde{T}(\omega) \coloneqq  \frac{\tilde{M}_{\epsilon}}{4 f(\omega) \max\{ F^{C,1}_{\nu}(\omega),  F^{C,-1}_{\nu}(\omega) \}}\min \left\{  \frac{1}{F_{\epsilon}^{U}(\omega) \Vert\Pi_{U_{\omega}\parallel C_{\omega}\oplus S_{\omega}}\Vert}, \frac{1}{F_{\epsilon}^{S}(\omega) \Vert\Pi_{S_{\omega}\parallel U_{\omega}\oplus C_{\omega}}\Vert} \right\}.
	\end{align*}
	We now redefine $\rho$ by setting
	\begin{align}\label{RHO}
		{\rho}(\omega) \coloneqq \min \left\{ \frac{1}{4} h^{-1}(\tilde{T}(\omega)), \frac{1}{4} h^{-1}(T(\omega) \right\}.
	\end{align}
	where $T(\omega)$ was defined in \eqref{eqn:T_omega}. Note that the implications of Proposition \ref{CENTER} remain valid for this new $\rho$. Define
	\begin{align*}
		\tilde{L}_{\varepsilon} \coloneqq \tilde{M}_{\varepsilon}\exp(\nu) \left\{ \frac{\exp(-\mu^++\epsilon+\nu)}{1-\exp(-\mu^++\epsilon+\nu)}+\frac{1}{1-\exp(\mu^-+\epsilon+\nu)} \right\}.
	\end{align*}
	Choosing $\varepsilon > 0$ sufficiently small gives
	\begin{align*}
		\Vert \Pi^{0}_{\omega}(\Gamma)-\Pi^{0}_{\omega}(\tilde{\Gamma})\Vert \geq \Vert v_{\omega} - \tilde{v}_{\omega}\Vert - \frac{5\tilde{L}_{\epsilon}}{4  \max\{ F^{C,1}_{\nu}(\omega),  F^{C,-1}_{\nu}(\omega) \}} \Vert\Gamma-\tilde{\Gamma}\Vert.
	\end{align*}
	Choosing $\tilde{M_{\varepsilon}}$ such that $5\tilde{L}_{\epsilon} \leq 1$ yields the first claim. The second follows by combining this estimate with Lemma \ref{lemma:I_lipschitz}.
\end{proof}

\begin{remark}
	Note that, by our assumptions,
	\begin{align*}
		\liminf_{n\rightarrow \pm \infty}\frac{1}{n}\log [\rho(\theta^{n}\omega)]\geq 0 .
	\end{align*}
	Indeed, this follows from the definition of $\rho$ and our assumption that $\limsup_{x\rightarrow 0}\frac{h(x)}{|x|^r}<\infty$.
\end{remark}

\begin{definition}\label{def:local_cocycle}
	For $\omega\in \Omega$ and  $\xi_{\omega}\in E_{\omega}$, set 
	\begin{align}
		\begin{split}
			{\phi}^{1}_{\omega}(Y_{\omega} + \xi_{\omega}) &\coloneqq Y_{\theta\omega}+\psi^{1}_{\omega}(\xi_{\omega})+ \delta \biggl( \frac{\Vert \xi_{\omega}\Vert}{\rho(\theta\omega)} \biggr) (\varphi^{1}_{\omega} (Y_{\omega}+\xi_{\omega})-Y_{\theta\omega}-\psi^{1}_{\omega}(\xi_\omega)) \quad \text{and} \\ \phi^{n}_{\omega}(Y_{\omega} + \xi_{\omega}) &\coloneqq \phi^{n-1}_{\theta\omega}\circ\phi^{1}_{\omega}(Y_{\omega}+\xi_{\omega}).
		\end{split}
	\end{align}
	
\end{definition}

\begin{remark}\label{remark:local_cocycle}
	Assume that for every $0\leq j\leq n$ it holds that $\Vert\varphi^{j}_{\omega}(Y_{\omega}+\xi_{\omega})-Y_{\theta^{j}\omega}\Vert\leq \rho(\theta^{j+1}\omega)$. Then $\phi^{i}_{\omega}(Y_{\omega}+\xi_{\omega})=\varphi^{i}_{\omega}(Y_{\omega}+\xi_{\omega})$ for every $0\leq i\leq n+1$. Therefore, $\phi$ is coincides with $\varphi$ locally up to a certain iteration and the number of iterations goes to infinity if $\Vert\xi_{\omega} \Vert\rightarrow 0$.
\end{remark}

\begin{definition}\label{definition_}
	Set
	\begin{align*}
		\mathcal{M}^{c,\nu}_{\omega}:=\lbrace\xi_{\omega}+Y_\omega\in E_{\omega}: \exists \Gamma \in{\Sigma}_\omega^{\nu} \ &\text{with} \ \Pi^{0}_{\omega}[\Gamma]=\xi_\omega\  \text{and} \\
		&\phi^{m}_{\theta^{n}\omega}(\Pi^{n}_{\omega}[\Gamma]+Y_{\theta^{n}\omega})=\Pi^{n+m}_{\omega}[\Gamma]+Y_{\theta^{n+m}\omega}\ \forall (m,n)\in\mathbb{N}\times \mathbb{Z} \rbrace.
	\end{align*} 
	We call this set the \emph{center manifold of the random map $\phi$.}
\end{definition}

\begin{lemma}\label{lemma:mane}
	With the same setting as in Proposition \ref{CENTER}, for $\Gamma\in \Sigma^{\nu}_{\omega}$,
	\begin{align}\label{center_}
		I_{\omega}(v_{\omega},\Gamma)=\Gamma \Longleftrightarrow \phi^{m}_{\theta^{n}\omega}(\Pi^{n}_{\omega}[\Gamma]+Y_{\theta^{n}\omega})=\Pi^{n+m}_{\omega}[\Gamma]+Y_{\theta^{n+m}\omega}, \ \ \forall (m,n)\in\mathbb{N}\times \mathbb{Z},
	\end{align}
	where
	\begin{align}\label{center_vector}
		\begin{split}
			& v_{\omega}=\Pi^{0}_{\omega}(\Gamma)+\sum_{j \geq 1}\big{[}\psi^{-j}_{\theta^{j}\omega}\Pi_{U_{\theta^{j}\omega}\parallel C_{\theta^{j}\omega}\oplus S_{\theta^{j}\omega}}\big{]}P_{\theta^{j-1}\omega,\rho}(\Pi^{j-1}_{\omega}[\Gamma]) -\\&\quad \sum_{j\leq 0}[\psi^{-j}_{\theta^{j}\omega}\circ\Pi_{S_{\theta^{j}\omega}\parallel U_{\theta^{j}\omega}\oplus C_{\theta^{j}\omega}}]P_{\theta^{j-1}\omega,\rho}(\Pi^{j-1}_{\omega}[\Gamma])\in C_{\omega}.
		\end{split}
	\end{align}
\end{lemma}

\begin{proof}
	The left to the right part, can be accomplished by using an induction argument on  $m$ for every fixed $n$. For the the other part of the claim, first note that 
	\begin{align*}
		&\Pi^{0}_{\omega}(\Gamma)+\\ &\sum_{j \geq 1}\big{[}\psi^{-j}_{\theta^{j}\omega}\Pi_{U_{\theta^{j}\omega}\parallel C_{\theta^{j}\omega}\oplus S_{\theta^{j}\omega}}\big{]}P_{\theta^{j-1}\omega,\rho}(\Pi^{j-1}_{\omega}[\Gamma]) -\sum_{j\leq 0}[\psi^{-j}_{\theta^{j}\omega}\circ\Pi_{S_{\theta^{j}\omega}\parallel U_{\theta^{j}\omega}\oplus C_{\theta^{j}\omega}}]P_{\theta^{j-1}\omega,\rho}(\Pi^{j-1}_{\omega}[\Gamma])=\\ &\Pi^{0}_{\omega}(\Gamma)+\sum_{j \geq 1}\big{[}\psi^{-j}_{\theta^{j}\omega}\Pi_{U_{\theta^{j}\omega}\parallel C_{\theta^{j}\omega}\oplus S_{\theta^{j}\omega}}\big{]}\big(\phi^{1}_{\theta^{j-1}\omega}(\Pi^{j-1}_{\omega}[\Gamma]+Y_{\theta^{j-1}\omega})-Y_{\theta^{j}\omega}-\psi^{1}_{\theta^{j-1}\omega}(\Pi^{j-1}_{\omega}[\Gamma])\big) -\\&\sum_{j\leq 0}\big[\psi^{-j}_{\theta^{j}\omega}\circ\Pi_{S_{\theta^{j}\omega}\parallel U_{\theta^{j}\omega}\oplus C_{\theta^{j}\omega}}\big] \big(\phi^{1}_{\theta^{j-1}\omega}(\Pi^{j-1}_{\omega}[\Gamma]+Y_{\theta^{j-1}\omega})-Y_{\theta^{j}\omega}-\psi^{1}_{\theta^{j-1}\omega}(\Pi^{j-1}_{\omega}[\Gamma])\big)=\\
		& \Pi^{0}_{\omega}(\Gamma)+\sum_{j \geq 1}\big{[}\psi^{-j}_{\theta^{j}\omega}\Pi_{U_{\theta^{j}\omega}\parallel C_{\theta^{j}\omega}\oplus S_{\theta^{j}\omega}}\big{]}\big( \Pi^{j}_{\omega}[\Gamma]-\psi^{1}_{\theta^{j-1}\omega}(\Pi^{j-1}_{\omega}[\Gamma])\big)-\\& \sum_{j\leq 0}\big[\psi^{-j}_{\theta^{j}\omega}\circ\Pi_{S_{\theta^{j}\omega}\parallel U_{\theta^{j}\omega}\oplus C_{\theta^{j}\omega}}\big]\big( \Pi^{j}_{\omega}[\Gamma]-\psi^{1}_{\theta^{j-1}\omega}(\Pi^{j-1}_{\omega}[\Gamma])\big)=\Pi^{0}_{\omega}(\Gamma)-\Pi_{U_{\theta^{j}\omega}\parallel C_{\theta^{j}\omega}\oplus S_{\theta^{j}\omega}}(\Pi^{0}_{\omega}(\Gamma))-\\& \Pi_{S_{\theta^{j}\omega}\parallel U_{\theta^{j}\omega}\oplus C_{\theta^{j}\omega}}(\Pi^{0}_{\omega}(\Gamma))=\Pi_{C_{\theta^{j}\omega}\parallel S_{\theta^{j}\omega}\oplus U_{\theta^{j}\omega}}(\Pi^{0}_{\omega}(\Gamma))\in  C_{\omega}.
	\end{align*}
	The same calculation for every $n\in\mathbb{Z}$ yields the left side of the claim.
\end{proof}
\begin{remark}\label{identity}
	Note that we also have
	\begin{align*}
		\phi^{n}_{\omega}(\Pi^{0}_{\omega}(\Gamma)+Y_{\omega})=Y_{\theta^{n}\omega}+\sum_{0\leq j\leq n-1}\psi^{n-1-j}_{\theta^{j+1}\omega}\big(P_{\theta^{j}\omega,\rho}(\phi^{j}_{\omega}(\Pi^{0}_{\omega}(\Gamma)+Y_{\omega})-Y_{\theta^j\omega})\big).
	\end{align*}  
\end{remark}
Finally, we can formulate the main result of this section.

\begin{theorem}[Center manifold theorem]\label{LCMT}
	Assume $(\Omega,\mathcal{F},\P,\theta)$ is an ergodic measure-preserving dynamical systems. Let $\varphi$ be a Fr\'echet-differentiable cocycle acting on a measurable field of Banach spaces $\{E_{\omega}\}_{\omega \in \Omega}$ and assume that $\varphi$ admits a stationary point $Y_{\omega}$. Furthermore, assume that the linearized cocycle $\psi$ around $Y$ is compact and satisfies Assumption \ref{ASSU} and the conditions of Theorem \ref{thm:MET_Banach_fields}. Suppose that at least one Lyapunov exponent is zero. For every $0<\nu<\min\lbrace \mu^{+},-\mu^{-}\rbrace$, let
	\begin{align*}
		h^{c}_{\omega} \colon C_\omega\rightarrow \mathcal{M}^{c,\nu}_{\omega}
	\end{align*}
	be defined as $h_{\omega}^{c}(v_{\omega}) = \Pi^{0}_{\omega}[\Gamma]$ where $\Gamma$ is the unique fixed point that satisfies $I_{\omega}(v_{\omega},\Gamma)=\Gamma$, and $I_{\omega}$ is the map defined in Proposition \ref{CENTER}. Let $\phi$ be defined as in Definition \ref{def:local_cocycle} for a random variable $\rho$ provided by Proposition \ref{CENTER}. Then the following statements hold:
	\begin{itemize}
		\item[(i)] $h^{c}_{\omega}$ is a homeomorphism, Lipschitz continuous and differentiable at zero.
		\item[(ii)] $\mathcal{M}^{c,\nu}_{\omega}$ is a topological Banach manifold\footnote{That means a $\mathcal{C}^0$ Banach manifold in the sense of definition \cite[3.1.3 Definition]{AMR88}.} modeled on $C_\omega$. Furthermore, for \( m \geq 1 \), assume that for every \(\omega \in \Omega\), the function \(\varphi^{1}_{\omega}\) is $m$-times Fréchet-differentiable and that the random norm \(\| \cdot \|_{E_\omega}\) is $m$-times Fréchet-differentiable outside of zero, i.e. that
			\begin{align*}
				&\| \cdot \|_{E_\omega} : E_\omega \setminus \{0\} \to \mathbb{R}, \\
				&\xi_{\omega} \mapsto \| \xi_\omega \|_{E_\omega}
			\end{align*}
			is $m$-times Fréchet-differentiable\footnote{Note that this is always true in the case of Hilbert spaces.}. Then this manifold is \( \mathcal{C}^{m-1} \) for every \(\omega \in \Omega\).
		\item[(iii)] $\mathcal{M}^{c,\nu}_{\omega}$ is $\phi$-invariant, i.e. for every $n \in \mathbb{N}_0$, $\phi^{n}_{\omega}(\mathcal{M}^{c,\nu}_{\omega})\subset\mathcal{M}^{c,\nu}_{\theta^n\omega}$.
	\end{itemize}
\end{theorem}

\begin{proof}
	
	\begin{itemize}
		\item[(i)] Lemma \ref{lemma:mane} immediately implies surjectivity of the map $h_{\omega}^c$. Let $v_{\omega}, \tilde{v}_{\omega} \in C_{\omega}$ and assume that $h_{\omega}^c(v_{\omega}) = h_{\omega}^c(\tilde{v}_{\omega})$. By definition, there exist $\Gamma, \tilde{\Gamma} \in \Sigma^{\nu}_{\omega}$ such that $I_{\omega}(v_{\omega},\Gamma) = \Gamma$,  $I_{\omega}(\tilde{v}_{\omega},\tilde{\Gamma)} = \tilde{\Gamma}$ and $\Pi^0_{\omega}[\Gamma] = \Pi^0_{\omega}[\tilde{\Gamma}]$. Lemma \ref{lemma:injecitvity} implies that $v_{\omega} = \tilde{v}_{\omega}$, thus $h_{\omega}^c$ is also injective. Lipschitz continuity of $h_{\omega}^c$ is a consequence of Lemma \ref{lemma:I_lipschitz} and continuity of the inverse follows again from Lemma \ref{lemma:injecitvity}.
		
		\item[(ii)] By the first item, $h^{c}_{\omega}$ defines a a continuous, Lipschitz chart on $\mathcal{M}^{c,\nu}_{\omega}$. Note that we modified the original cocycle $\varphi$ by $\phi$. Let \(\varphi^{1}_{\omega}\) be \(\mathcal{C}^m\) for every \(\omega \in \Omega\) and let the random norms \(\| \cdot \|_{E_\omega}\) be smooth outside of zero for every \(\omega \in \Omega\). From Definition \ref{def:local_cocycle}, it follows that the function \(\phi^1_{\omega}\) is also \(\mathcal{C}^m\) for every \(\omega \in \Omega\) in \(E_\omega \setminus \{ Y_\omega \}\).
			At \(Y_\omega\), since in a neighborhood of zero we have \(\delta \left( \frac{\| \xi_\omega \|}{\rho(\theta \omega)} \right) = 1\), it follows that \(\phi^1_{\omega}\) is also \(\mathcal{C}^m\) at \(Y_\omega\). Recall that we construct the center manifold for the cocycle \(\phi\) using the implicit function theorem for the map \(I\). Therefore, when for every \(\omega \in \Omega\), \(\phi^{1}_{\omega}\) is \(\mathcal{C}^m\), the implicit function theorem implies that the center manifold is \(\mathcal{C}^{m-1}\). 
		\item[(iii)] The $\phi$-invariance property of $\mathcal{M}^{c,\nu}_{\omega}$ follows from Definition \ref{definition_} and Lemma \ref{lemma:mane}. Indeed, assume $\xi_\omega+Y_\omega\in\mathcal{M}^{c,\nu}_{\omega}$, then for $\Gamma\in \Sigma^{\nu}_{\omega}$ and $v_\omega$ defined in \eqref{center_vector}, we have $I_{\omega}(v_\omega,\Gamma)=\Gamma$ and $\Pi^{0}_{\omega}[\Gamma]=\xi_\omega$. Let us define  $\tilde{\Gamma}\in \Sigma^{\nu}_{\theta\omega}$, with $\Pi^{n}_{\theta\omega}[\tilde{\Gamma}]:=\Pi^{n+1}_{\omega}[{\Gamma}]$ and $\tilde{v}_{\theta\omega}\in C_{\theta\omega}$ by $\tilde{v}_{\theta\omega}=\psi^{1}_{\omega}(v_\omega)$. Note that $\Pi^{0}_{\theta\omega}[\tilde{\Gamma}]=\phi^{1}_{\omega}(\xi_\omega+Y_\omega)-Y_{\theta\omega}$ and $I_{\theta\omega}(\tilde{v}_{\theta\omega},\tilde{\Gamma})=\tilde{\Gamma}$. This proves the claim.
		
		%\begin{align*}
		%&\tilde{v}_{\theta\omega}=\Pi^{0}_{\theta\omega}(\tilde{\Gamma})+\sum_{j \geq 1}\big{[}\psi^{-j}_{\theta^{j+1}\omega}\Pi_{U_{\theta^{j+1}\omega}\parallel C_{\theta^{j+1}\omega}\oplus S_{\theta^{j+1}\omega}}\big{]}P_{\theta^{j}\omega,\rho}(\Pi^{j-1}_{\theta\omega}[\tilde{\Gamma}]) -\\&\quad \sum_{j\leq 0}[\psi^{-j}_{\theta^{j+1}\omega}\circ\Pi_{S_{\theta^{j+1}\omega}\parallel U_{\theta^{j+1}\omega}\oplus C_{\theta^{j+1}\omega}}]P_{\theta^{j}\omega,\rho}(\Pi^{j-1}_{\theta\omega}[\tilde{\Gamma}])\in C_{\theta\omega}.
		%\end{align*}
		%\phi^{m}_{\theta^{n+1-m}\omega}(\Pi^{n+1-m}_{\omega}[{\Gamma}])+Y_{\theta^{n+1}\omega}
	\end{itemize}
\end{proof}
	We stated the invariance property of the manifold $\mathcal{M}^{c,\nu}_{\omega}$ for the modified system $\phi$ in the last theorem. In the following corollary, we will reformulate it for the original cocycle $\varphi$.
	\begin{corollary}
		Assume $\mathcal{M}^{c,\nu}_{\omega}$ to be the topological manifolds that come from Theorem \ref{LCMT}. Assume further that $\rho$ is the  random variable  that is provided by Proposition \ref{CENTER}. Then the following invariance holds for $\varphi$:
		\begin{align}\label{RRFR}
			\varphi^{1}_{\omega}\left(\mathcal{M}^{c,\nu}_{\omega} \cap \left\{ Y_\omega + \xi_\omega\, :\,  \Vert\xi_\omega\Vert \leq \rho(\theta\omega) \right\}\right) \subseteq \mathcal{M}^{c,\nu}_{\theta\omega}.
		\end{align}
		Also, more generally, for
		\[ D^{n}_{\omega} = \bigcap_{0 \leq j \leq n-1} \left\{ Y_\omega + \xi_\omega\, :\,  \|\varphi^{j}_{\omega}(Y_\omega + \xi_\omega) - Y_{\theta^{j}\omega}\| \leq \rho(\theta^{j+1}\omega) \right\}, \]
		we have
		\begin{align}\label{GVBVV}
			\varphi^{n}_{\omega} \left( \mathcal{M}^{c,\nu}_{\omega} \cap D^{n}_{\omega} \right) \subseteq \mathcal{M}^{c,\nu}_{\theta^n\omega}.
		\end{align}
	\end{corollary}
	\begin{proof}
		Recall from Remark \ref{remark:local_cocycle} that $\phi$ and our original cocycle $\varphi$ coincide locally around the stationary point. More precisely, from Definition \ref{def:local_cocycle}, if $\| \xi_\omega \| \leq \rho(\theta \omega)$, then $\phi^1_{\omega}(Y_{\omega} + \xi_\omega) = \varphi^1_{\omega}(Y_{\omega} + \xi_\omega)$. In particular, for every $z_\omega \in \mathcal{M}^{c,\nu}_{\omega} \cap \{ Y_\omega + \xi_\omega\, :\,  \| \xi_\omega \| \leq \rho(\theta \omega) \}$, we have that $\phi^1_{\omega}(z_\omega) = \varphi^1_{\omega}(z_\omega)$. From Theorem \ref{LCMT}, item (iii), $\mathcal{M}^{c,\nu}_{\omega}$ is $\phi$-invariant. Therefore, we can conclude \eqref{RRFR}. To prove \eqref{GVBVV}, we can argue similarly and use Remark \ref{remark:local_cocycle}.
	\end{proof}
	\begin{remark}
		Let us motivate the definition of the map $I_{\omega}$ defined in the proof of Theorem \ref{LCMT} a bit more. In fact, the main idea for proving existence of the center manifold in our setting is very similar to the standard method involving the Lyapunov-Perron map. This approach is frequently used in the literature to demonstrate the existence of invariant manifolds in both deterministic and stochastic contexts. Recall that \(\varphi^{n}_{\omega}(Y_{\omega})=Y_{\theta^n\omega}\) and \(D_{Y_{\theta^{j}\omega}}\varphi^{n-j} = \psi^{n-j}_{\theta^{j}\omega}\), so that for every \(n \geq 1\),
		\begin{align*}
			\varphi^{n}_{\omega}(Y_{\omega}+\xi_\omega)-\varphi^{n}_{\omega}(Y_\omega)=\psi^{n}_{\omega}(\xi_\omega)+\sum_{1\leq j\leq n}\psi^{n-j}_{\theta^{j}\omega}\left(P_{\theta^{j-1}\omega}(\underbrace{\varphi^{j-1}_{\omega}(Y_{\omega}+\xi_\omega)-\varphi^{j-1}_{\omega}(Y_\omega)}_{w^{j-1}_{\omega}})\right).
		\end{align*}
		Remember that
		\begin{align*}
			w^{j}_{\omega}=\Pi_{U_{\theta^{j}\omega}\parallel C_{\theta^{j}\omega}\oplus S_{\theta^{j}\omega}}(w^{j}_{\omega})+\Pi_{C_{\theta^{j}\omega}\parallel S_{\theta^{j}\omega}\oplus U_{\theta^{j}\omega}}(w^{j}_{\omega})+\Pi_{S_{\theta^{j}\omega}\parallel U_{\theta^{j}\omega}\oplus C_{\theta^{j}\omega}}(w^{j}_{\omega}).
		\end{align*} 
		We keep the center part and assume that our map is also defined on the negative iterations and that the following relations hold for the stable and unstable parts: 
		\begin{align}\label{FCVV}
			\sum_{j\geq 1}\big{[}\psi^{n-j}_{\theta^{j}\omega}\circ\Pi_{U_{\theta^{j}\omega}\parallel C_{\theta^{j}\omega}\oplus S_{\theta^{j}\omega}}\big{]}  P_{\theta^{j-1}\omega}\big{(}w^{j-1}_{\omega}\big{)}=\sum_{j\leq 0}[\psi^{n-j}_{\theta^{j}\omega}\circ\Pi_{S_{\theta^{j}\omega}\parallel U_{\theta^{j}\omega}\oplus C_{\theta^{j}\omega} }]P_{\theta^{j-1}\omega}(w^{j-1}_{\omega}).
		\end{align}
		This assumption is reasonable. Indeed, first note that on the left side of \eqref{FCVV}, for \(\psi^{n-j}_{\theta^j\omega}\), the term \(n-j\) is negative for \(j > n\). Therefore, since we compose this function with \(\Pi_{U_{\theta^{j}\omega} \parallel C_{\theta^{j}\omega} \oplus S_{\theta^{j}\omega}}\), we can expect, from item (iii) in Theorem \ref{thm:MET_Banach_fields}, that the infinite sum is convergent. Similarly, we can argue for the right part. The identity assumption aligns with the idea that, in the center manifold, we expect the unstable and stable parts to neutralize each other, with only the center part coming into play. These explanations motivate the definitions of \(\Sigma_{\omega}^{\nu}\) and \(I_{\omega}\). As stated earlier, the main idea is well-established through what is known as the \emph{Lyapunov-Perron method}.
	\end{remark}
\section{Application}\label{app}
To emphasize the generality of our main result, in this part, we illustrate it with several examples. Since our examples are relatively independent, we only briefly go into the details and refer to the original papers for more background.
\subsection{Application to stochastic singular delay equations }\label{A_1}
We first show how our result can applied to a large family of stochastic delay equations (SDDEs). In its simplest formulation, an SDDE takes the form 
\begin{align}\label{eqn:sdde_intro}
	\begin{split}
		dY_t &= L(Y_t,Y_{t-r})\, \mathrm{d}t + G(Y_t,Y_{t-r})\, \mathrm{d}B_t(\omega),\\
		Y_t &= \xi_t; \quad t \in [-r,0].
	\end{split}
\end{align} 
Here $r > 0$ and $B=(B^{1},B^{2},...,B^{d})$, such that the components of $B$ are independent Brownian motions. Also, $L:\mathbb{R}^m \times \mathbb{R}^m\rightarrow \mathbb{R}^m$ is a linear drift and $G:\mathbb{R}^m\times \mathbb{R}^m\rightarrow \mathcal{L}(\mathbb{R}^d,\mathbb{R}^m)$ is either linear or is a bounded differentiable function up to four times with bounded derivatives. The main obstacle to developing a dynamical theory for this type of equations is that while we can solve them, the solutions necessarily do not depend continuously on the initial values. Therefore, we can not expect to define a classical semi-flow on a reasonable Banach space, cf.\cite{Mao08, Moh84} and \cite [Theorem 0.2.]{GVRS22}. However, to sort out this problem in the recent papers \cite{GVRS22, GVR21}, the authors proved it is possible to find a measurable field of Banach spaces $(\{E_{\omega}\}_{\omega \in \Omega},\Delta)$ and an ergodic measure-preserving dynamical system  $(\Omega,\mathcal{F},\P,\theta)$, such that the family of solutions to equation \ref{eqn:sdde_intro}, defines a continuous cocycle on $\{E_{\omega}\}_{\omega \in \Omega}$. To be more precise (while losing some minor formalities), for $\frac{1}{3}<\beta\leq \frac{1}{2}$ and every $\omega\in\Omega$ if we set  $E_{\omega}=\mathscr{D}_{B(\omega)}^{\beta}([-r,0])$, then this family constitutes a measurable fields of Banach spaces, cf.\cite[Theorem 1.1]{GVRST22}. Here by $\mathscr{D}_{B(\omega)}^{\beta}([-r,0])$, we mean the space of \emph{controlled rough paths}, cf.\cite [Definition 4.6]{FH20}. It is proven in \cite[Section 2]{GVR21} that the solutions to equation \ref{eqn:sdde_intro}, act on $\lbrace \mathscr{D}_{B(\omega)}^{\beta}([-r,0])\rbrace_{\omega \in\Omega}$. Let $G(0,0)=0$, then $Y=0$ is a stationary solution to \eqref{eqn:sdde_intro}. From \cite [Lemma 3.2.]{GVR21}, the linearized cocycle around the stationary point defines a compact linear cocycle on $\lbrace E_\omega\rbrace_{\omega\in\Omega}$ and the semi-invertible Multiplicative Ergodic Theorem holds. Assume further that one of these Lyapunov exponents is equal to zero. This, for example, holds when $ \frac{\partial G}{\partial x}(0,0)=\frac{\partial G}{\partial y}(0,0)=0$ and the spectrum of $L$ contains at least one element on the unit sphere. This can be shown by proving slightly generalized versions of \cite[Proposition 2.9 and Proposition 3.7]{GVR21} that can incorporate a linear drift. For $\xi\in E_\omega$, let $(Y^{\xi}_t)_{t\geq 0}$ be our unique global solution and for $n\in\mathbb{N}$ set $\varphi^{n}_{\omega}(\xi)=(Y^{\xi}_t)_{(n-1)r\leq t\leq nr}\in E_{\theta^n\omega}$. Then from \cite[Page 17]{GVR21}, Assumption \ref{ASSU} holds. Therefore, we can apply Theorem \ref{LCMT} to deduce the existence of center manifolds for this family of equations.

\subsection{Application to rough differential equations}\label{A_2}
For $\frac{1}{4}<\gamma\leq \frac{1}{2}$, we consider the following rough equation
\begin{align}\label{SDE}
	\mathrm{d}Z_{t} = V(Z_t)\, \mathrm{d} \mathbf{X}_t(\omega)+V_{0}(Z_t)\, \mathrm{d}t, \ \ \  Z_{0}=z_0\in \mathbb{R}^m.
\end{align}
Assume that $(\Omega,\mathcal{F},\P,(\theta_t)_{t \in \R})$ is an ergodic measure-preserving dynamical system
%, cf.\cite[Definition 2.1]{GVR23A} 
and that the process $\mathbf{X}:\mathbb{R}\times\Omega\rightarrow T^{3}(\mathbb{R}^d)$ is a $\frac{1}{\gamma}$-variation \emph{geometric rough path cocycle}, cf. \cite [Definition 2]{RS17}. Typical examples of such processes are fractional Brownian motions.
For some background about this equation and the setting of solution, we refer to \cite{FH20}, \cite{RS17}, and \cite{GVR23A}. The first paper that dealt with the center manifold for this type of equation is \cite{NK21}. By imposing standard assumptions on $V$ and $V_0$, cf. \cite[Assumption 1.1., Example 1.2.]{GVR23A}, the equation \eqref{SDE} for every initial value admits a unique and global solution. Let us fix an arbitrary time step $t_0>0$ and for $n\in\mathbb{N}$, define $\tilde{\theta}^{n}:=\theta_{nt_0}$. Assume $(\phi(t,\mathbf{X}(\omega),z_0))_{t\geq 0}$, is the solution to equation \ref{SDE}. We set $\varphi^{n}_{\omega}(z_0):=\phi(nt_0,\mathbf{X}(\omega),z_0)$. From our assumptions on $\mathbf{X}(\omega)$,
\begin{align*}
	n,m\in\mathbb{N}:\ \  \varphi^{m+n}_{\omega}(z_0)=\varphi^{n}_{\tilde{\theta}^{m}\omega}\big(\varphi^{m}_{\omega}(z_0)\big).
\end{align*}
In \cite{GVR23A}, this family of equations is studied in detail. Also, several examples in which Assumption \ref{ASSU} holds are given. The simplest case is similar to the previous one when $V(0)=D_{0}V=0$ and $D_{0}V_{0}$ has at least one eigenvalue on the unit sphere. We can then apply Theorem \ref{LCMT} to deduce the existence of center manifolds for $\varphi$.
\subsection{Application to stochastic partial differential equations with multiplicative rough noise}\label{A_3}
Our next example is the following equation,
\begin{align}\label{SPDE_EQU}
	\mathrm{d}Z_t = AZ_t \, \mathrm{d}t+F(Z_t) \, \mathrm{d}t+G(Z_t)\circ\mathrm{d}\mathbf{X}_t, \ \ \ \  Z_{0}=z_0\in \mathcal{B}_0=\mathcal{B},
\end{align} 
where $\mathcal{B}$ is a separable Banach space. We again assume that $(\Omega,\mathcal{F},\P,(\theta_t)_{t \in \R})$ is an ergodic measure-preserving dynamical system and for $\frac{1}{3}<\gamma\leq \frac{1}{2}$, the process $\mathbf{X}:\mathbb{R}\times\Omega\rightarrow T^{3}(\mathbb{R}^d)$ is a $\frac{1}{\gamma}$-variation \emph{geometric rough path cocycle}. We assume further that we have a \emph{monotone
	family of interpolation spaces} $\lbrace (\mathcal{B}_{\beta},\vert\ \cdot \vert_{\beta})\rbrace_{\beta\in\mathbb{R}}$, cf.\cite[Definition 2.1]{GHT21}. This family of equations are studied first in \cite{GH19}. For more details about this equation, we refer to \cite{GVR23B}. As before we use $(\phi(t,\mathbf{X}(\omega),z_0))_{t\geq 0}$, to denote the solution. We fix an arbitrary time step $t_0>0$ also for every $n\in\mathbb{N}$, define $\tilde{\theta}^{n}:=\theta_{nt_0}$ and $\varphi^{n}_{\omega}(z_0):=\phi(nt_0,\mathbf{X}(\omega),z_0)$. In \cite{GVR23B}, the authors proved that by imposing some natural conditions on $\mathbf{X}(\omega)$ (cf.\cite[Assumption 2.8.]{GVR23B}), Assumption \ref{ASSU} holds, cf.\cite [Theorem 3.4, Page 36]{GVR23B}. We can then deduce the existence of center manifolds around the stationary point for $\varphi$. Similar to the previous examples, a simple case is when zero belongs to the spectrum of $A$ and $F(0)=D_{0}F=G(0)=D_{0}G=0$.

\subsection{Application to stochastic partial differential equations with fractional noise}\label{A_4}
For $H>\frac{1}{2}$, we now consider the following equation
\begin{align}\label{SPDE_EQU_1}
	\mathrm{d}Z_t = AZ_t \, \mathrm{d}t+F(Z_t) \, \mathrm{d}t+G(Z_t)\circ\mathrm{d}B^{H}_t, \ \ \ \  Z_{0}=z_0\in \mathcal{B}_0=\mathcal{B}.
\end{align} 
Assume $\mathcal{B}$ is a separable Hilbert space and $\lbrace e_i\rbrace_{i\geq 1}$ a complete orthonormal basis for it. Furthermore, $\lbrace \mathcal{\beta}^{H}_{i}\rbrace_{i\geq 1}$ is a sequence of stochastically independent one-dimensional fractional Brownian motions and  $\lbrace\lambda_i\rbrace_{i\geq 1}$ a sequence of non-negative numbers such that $\sum_{i\geq 1}{\lambda_i}<\infty$. Also,
\begin{align*}
	t\in\mathbb{R}: \ \  B^{H}_t=\sum_{i\geq 1}\sqrt{\lambda_i} \mathcal{\beta}^{H}_{i}(t)e_i.
\end{align*}
	This kind of problem was first studied in \cite{GKS10} under the stronger assumption \(\sum_{i \geq 1} \sqrt{\lambda_i} < \infty\) to address the existence of the unstable manifold. More recently, the existence of the stable manifold has been addressed in \cite{LNZ23}. We refer to these papers for a solid background on this equation.
	Since \(H > \frac{1}{2}\), the stochastic integral can be defined path-wise using Young's theory (cf.\ \cite[Section 8.3]{FH20}). Therefore, in contrast to \eqref{SPDE_EQU}, we do not need the iterated integrals (i.e., \(\int \beta^{H}_{i} \, \mathrm{d} \beta^{H}_{j}\)), and the stochastic integral is defined by
	\[
	\int_{0}^{t} S(t-s) G(Z_s) \circ \mathrm{d} B^{H}_\tau = \sum_{i \geq 1} \sqrt{\lambda_i} \int_{0}^{t} S(t-s) G(Z_s) e_i \, \mathrm{d} \mathcal{\beta}^{H}_{i}(\tau).
	\]
	For this example, we strongly believe that by adapting the techniques developed in \cite{GVR23B}, we can verify Assumption \ref{ASSU}. Alternatively, using fractional calculus as in \cite{LNZ23}, this assumption also seems to be verifiable. 
	%Indeed, compared to \eqref{SPDE_EQU}, the noise we consider is more regular, so we do not require the L\'evy areas. 
	We also believe that the techniques developed in \cite{GVR23B} can be used to establish the existence of stable and unstable manifolds, thus generalizing the results of \cite{GKS10} and \cite{LNZ23} without relying on fractional calculus.
	\begin{remark}
		The main strategy to verify Assumption \ref{ASSU} usually involves the following steps:
		\begin{enumerate}
			\item \textbf{Obtain an integrable a priori bound.} First, obtain an a priori bound for the solution in terms of the noise and initial value. It is important that this a priori bound is integrable with respect to the noise.
			\item \textbf{Linearize  the equations.} Next, linearize the equations and use the a priori bound to estimate the linearized equations. This step ensures that the conditions of the Multiplicative Ergodic Theorem (MET) are satisfied.
			\item \textbf{Estimate the differences.} Finally, estimate the difference between solutions to the linearized equations with two different initial values to verify \eqref{eqn:diff_bound_P}.
		\end{enumerate}
		We refer to \cite{GVR21}, \cite{GVR23A}, and \cite{GVR23B} for practical examples of how these steps are implemented.
	\end{remark}
% From Lemma \ref{lemma:mane}, it follows that $h_{\omega}^c$ is bijective. We proceed showing the Lipschitz property. Let a contnous  $v_{\omega}, \tilde{v}_{\omega} \in C_{\omega}$. From Lemma \ref{lemma:mane}, there are $\Gamma, \tilde{\Gamma} \in \Sigma^{\nu}_{\omega}$ such that

% \subsection*{Acknowledgements}
% \label{sec:acknowledgements}

% MGV acknowledges a scholarship from the Berlin Mathematical School (BMS). SR acknowledges financial support by the DFG via Research Unit FOR 2402.

\bibliographystyle{alpha}
\bibliography{refs}

\end{document}